\documentclass{article}
\usepackage{ijcai17}

\usepackage{times}
\usepackage{color}
\usepackage{algorithm}
\usepackage{algorithmic}
\usepackage{graphicx}
\usepackage{subfigure} 
\usepackage{amsmath}
\usepackage{amsthm}
\usepackage{amssymb}
\usepackage{enumitem}
\usepackage{multirow}
\usepackage{url}

\usepackage{array}
\newcolumntype{L}[1]{>{\raggedright\let\newline\\\arraybackslash\hspace{0pt}}m{#1}}
\newcolumntype{C}[1]{>{\centering\let\newline  \\\arraybackslash\hspace{0pt}}m{#1}}
\newcolumntype{R}[1]{>{\raggedleft\let\newline \\\arraybackslash\hspace{0pt}}m{#1}}

\newtheorem{theorem}{Theorem}[section]
\newtheorem{lemma}[theorem]{Lemma}
\newtheorem{proposition}[theorem]{Proposition}
\newtheorem{corollary}[theorem]{Corollary}

\newtheorem{definition}{Definition}[section]

\newcommand{\Px}[2]{\text{prox}_{#1}(#2) } 
\newcommand{\Gm}[1]{\mathcal{G}(#1) }
\newcommand{\rank}{\text{rank}}

\newcommand{\TV}[1]{\,\text{TV}(#1)}

\newcommand{\SO}[1]{P_{\Omega}(#1) }

\newcommand{\NM}[2]{\| #1 \|_{#2}}
\newcommand{\R}{\mathbb R}

\title{Efficient Inexact Proximal Gradient Algorithm for Nonconvex Problems}

\author{\normalsize Quanming Yao$^1$ \; James T. Kwok$^1$ \; Fei Gao$^2$ \; Wei Chen$^2$ \; Tie-Yan Liu$^2$ \\
	\normalsize $^1$Department of Computer Science and Engineering\\
	\normalsize Hong Kong University of Science and Technology \\
	\normalsize Clear Water Bay, Hong Kong\\
	\normalsize \{qyaoaa, jamesk\}@cse.ust.hk\\
	\normalsize $^2$Microsoft Research, Beijing 100080, China\\
	\normalsize \{feiga, wche, tyliu\}@microsoft.com}

\begin{document}

\maketitle

\begin{abstract}
The proximal gradient algorithm has been popularly used for convex optimization.
Recently, it has also been 
extended for nonconvex problems, and
the current state-of-the-art
is the nonmonotone accelerated proximal gradient 
algorithm.
However,
it typically requires two exact proximal steps in each iteration,
and can be inefficient when the proximal step 
is expensive.
In this paper, we propose an efficient proximal gradient algorithm 
that requires only one 
inexact (and thus less expensive)
proximal step
in each iteration.
Convergence to a critical point 
is still guaranteed
and has a $O(1/k)$ convergence rate, which is the best rate for nonconvex problems with
first-order methods.
Experiments on a number of problems demonstrate that the proposed algorithm has comparable 
performance as the state-of-the-art, but is much faster.
\end{abstract}

\section{Introduction}
\label{sec:intro}

%

In regularized risk minimization, we consider optimization problems of the form 
\begin{equation} \label{eq:compfunc}
\min_x F(x) \equiv f(x) + g(x),
\end{equation}
where $f$ is the loss, and $g$ is the regularizer.
Typically, $f$ is smooth and convex (e.g., square and logistic losses),
and $g$ is convex but may not be differentiable (e.g., $\ell_1$ and nuclear norm regularizers).
The proximal gradient (PG) algorithm \cite{parikh2014proximal}, 
together with its accelerated variant (APG) \cite{beck2009fast,nesterov2013gradient},
have been popularly used for solving this convex problem.
Its crux is the proximal step
$\Px{g}{\cdot} = \arg\min_x  \frac{1}{2}\NM{x - \cdot}{2}^2 + \eta g(x)$,
which can often be easily computed
in closed-form.

While convex  regularizers are easy to use,
the resultant predictors may be biased \cite{zhang2010analysis}.
Recently, there is growing interest in the use of nonconvex regularizers, 
such as the log-sum-penalty \cite{candes2008enhancing} and capped $\ell_1$-norm
\cite{zhang2010analysis} regularizers.
It has been shown that these often lead to sparser and more accurate models
\cite{gong2013gist,canyi2014,zhong2014gradient,yao2015fast}.
However, the associated proximal steps become more difficult to compute analytically, and 
cheap closed-form solutions exist only for some simple nonconvex regularizers
\cite{gong2013gist}.
This is further aggravated by the fact that 
the state-of-the-art PG algorithm for nonconvex optimization, namely the nonmonotone accelerated proximal gradient (nmAPG) algorithm \cite{li2015accelerated},
needs more than one proximal steps in each iteration.

When the optimization objective is convex,
one can reduce the computational complexity 
of the proximal step 
by only computing
it inexactly (i.e., approximately).
Significant speedup has been observed in practice, and
the resultant inexact PG algorithm has the same convergence guarantee as the exact algorithm
under mild conditions
\cite{schmidt2011convergence}.
However,  
on nonconvex problems,
the use of inexact proximal steps
has
not been explored.
Moreover, convergence of nmAPG hinges on the use of 
exact proximal steps.

In this paper,
we propose 
a new PG algorithm 
for nonconvex problems.
Unlike nmAPG,
it performs only one proximal step
in each iteration. Moreover, 
the proximal step
can be inexact.
The algorithm is guaranteed to converge to a critical point of the nonconvex objective.
Experimental results 
on nonconvex total variation models and nonconvex low-rank matrix learning 
show that 
the proposed algorithm is much faster than 
nmAPG and other state-of-the-art,
while still producing solutions of comparable quality.

The rest of the paper is organized as follows. 
Section~\ref{sec:rel} provides a brief review on the PG algorithm and its accelerated
variant.
The proposed algorithm is described in Section~\ref{sec:alg},
and its convergence analysis studied in Section~\ref{sec:conv}.
Experimental results are presented in Section~\ref{sec:expt}, 
and the last section gives some concluding remarks.


\vspace{-5px}

\section{Related Work}
\label{sec:rel}

In this paper,
we assume that $f$ in (\ref{eq:compfunc}) is $L$-Lipschitz smooth (i.e.,
$\NM{\nabla f(x) - \nabla f(y)}{2} \le L \NM{x - y}{2}$), 
and $g$ is proper, lower semi-continuous.
Besides, $F=f+g$ in (\ref{eq:compfunc}) is bounded from below,
and $\lim_{\NM{x}{2} \rightarrow \infty} F(x) = \infty$.
Moreover, both $f$ and $g$ can be nonconvex.

First, we consider the case where $f$ and $g$ 
in (\ref{eq:compfunc})
are convex.
At 
iteration
$k$, the
accelerated proximal gradient (APG) algorithm generates $x_{k + 1}$ as
\begin{eqnarray} 
y_k & = & x_k + \theta_k (x_k - x_{k - 1}), 
\label{eq:yk}
\\
x_{k + 1} & = & \Px{\eta g}{y_k - \eta \nabla f(y_k)}, 
\label{eq:grad}
\end{eqnarray} 
where 
$\theta_k = \frac{k - 1}{k + 2}$
and 
$\eta$ is the stepsize \cite{beck2009fast,nesterov2013gradient}.
When $\theta_k = 0$, APG reduces to the plain PG algorithm.

On nonconvex problems,
$y_k$ can be a bad extrapolation and
the iterations in \eqref{eq:yk}, \eqref{eq:grad} 
may 
not converge
\cite{beck2009tv}.
Recently, a number of PG extensions have been proposed to alleviate this
problem.
The iPiano \cite{ochs2014ipiano}, NIPS \cite{ghadimi2015accelerated}, 
and UAG \cite{ghadimi2015accelerated} algorithms
allow $f$ to be nonconvex, but still requires $g$ to be convex.
The GD algorithm \cite{attouch2013convergence} also allows 
$g$ to be nonconvex, but does not support acceleration.

The current state-of-the-art is the
nonmonotone APG (nmAPG) algorithm\footnote{A less efficient
monotone APG (mAPG) algorithm is also proposed in
\cite{li2015accelerated}.}
\cite{li2015accelerated},
shown in Algorithm~\ref{alg:nmapg}.
It allows both $f$ and $g$ to be nonconvex, and also uses acceleration.
To guarantee
convergence, nmAPG 
ensures that
the objective 
is 
sufficiently reduced
in each iteration:
\begin{align}
F(x_{k + 1})
\le F(x_k) - \frac{\delta}{2}\NM{v_{k + 1} - x_k}{2}^2,
\label{eq:cond}
\end{align}
where 
$v_{k + 1} = \Px{\eta g}{x_k - \eta \nabla f(x_k)}$
and 
$\delta > 0$ is a constant.
A second proximal step has to be performed (step~8)
if 
a variant of \eqref{eq:cond}
is not met
(step~5).



\begin{algorithm}[ht]
\caption{Nonmonotone APG (nmAPG).}
\begin{algorithmic}[1]
	\REQUIRE choose $\eta \in (0, 1/L)$, a positive constant $\delta$,
	$\Delta_1 = F(x_1)$, $q_1 = 1$, and $\nu \in (0, 1)$;
	\STATE $x_0 = x_1 = x^a_1 = 0$ and $t_0 = t_1 = 1$;
	\FOR{$k = 1, \dots, K $}
	\STATE $y_k = x_k + \frac{t_{k-1}}{t_k}(x^a_k - x_{k - 1})+\frac{t_{k-1}-1}{t_k}(x_k - x_{k - 1})$;
	\STATE $x^a_{k + 1} = \Px{\eta g}{y_k - \eta  \nabla f(y_k)}$;
	\IF{$F(x^a_{k + 1}) \le \Delta_k - \frac{\delta}{2}\NM{x^a_{k + 1} - y_k}{2}^2$}
	\STATE $x_{k + 1} = x^a_{k + 1}$;
	\ELSE
	\STATE $x^p_{k + 1} = \Px{\eta g}{x_k - \eta  \nabla f(x_k)}$;
	\STATE 
	$x_{k + 1} =
	\begin{cases}
	x^a_{k + 1}
	& F(x^a_{k + 1}) \le F(x^p_{k + 1})
	\\
	x^p_{k + 1}
	& \text{otherwise}
	\end{cases}
	$;
	\ENDIF 
	\STATE $q_{k + 1} = \nu q_k + 1$;
	\STATE $t_{k + 1} = \frac{1}{2}\left((4 t_k^2 + 1)^{1/2} + 1\right)$;
	\STATE $\Delta_{k + 1} = \frac{1}{q_{k + 1}} ( \nu q_k \Delta_k + F(x_{k + 1}) )$;
	\ENDFOR
	\RETURN $x_{K + 1}$.
\end{algorithmic}
\label{alg:nmapg}
\end{algorithm}


\vspace{-10px}

\section{Efficient APG for Nonconvex Problems}
\label{sec:alg}

The proposed algorithm is shown in Algorithm~\ref{alg:ours}. 
Following \cite{schmidt2011convergence},
we use a simpler acceleration scheme in step~3.
Efficiency of the algorithm comes from two key ideas: 
reducing the number of proximal steps to one in each iteration (Section~\ref{sec:reduce}); and
the use of inexact proximal steps (Section~\ref{sec:iextpx}).
Besides,
we also allow nonmonotone update on the objective 
(and so $F(y_k)$ may be larger than $F(x_k)$).
This helps to jump from narrow curved valley and improve convergence \cite{grippo2002nonmonotone,wright2009sparse,gong2013gist}.
Note that 
when the proximal step is inexact,
nmAPG does not guarantee convergence 
as its Lemma~2 
no longer holds.

\begin{algorithm}[ht]
\caption{Noconvex inexact APG (niAPG) algorithm.}
\begin{algorithmic}[1]
	\REQUIRE choose $\eta \in (0, \frac{1}{L})$ and $\delta \in (0, \frac{1}{\eta} - L)$;
	\STATE $x_0 = x_1 = 0$;
	\FOR{$k = 1, \dots, K $}
	\STATE $y_k = x_k + \frac{k - 1}{k + 2} (x_k - x_{k - 1})$;
	\STATE $\Delta_k = \max_{t = \max(1, k - q), \dots, k} F(x_t)$;
	\IF{$F(y_k) \le \Delta_k$}
	\STATE $v_k = y_k$;
	\ELSE
	\STATE $v_k = x_k$;
	\ENDIF
	\STATE $z_k = v_k - \eta \nabla f(v_k)$;
	\STATE $x_{k + 1} = 
	\Px{\eta g}{z_k}$;
	\text{                // possibly inexact} 
	\ENDFOR 
	\RETURN $x_{K + 1}$.
\end{algorithmic}
\label{alg:ours}
\end{algorithm}


\vspace{-10px}

\subsection{Using Only One Proximal Step}
\label{sec:reduce}

Recall that 
on extending APG to nonconvex problems, 
the key is to ensure a sufficient decrease of the objective in each iteration.
Let $\rho=1/\eta$.
For the standard PG algorithm 
with exact proximal steps,
the decrease in $F$ can be bounded as follows.

\begin{proposition}[\cite{gong2013gist,attouch2013convergence}]
	\label{pr:exactpx}
	$F(x_{k + 1}) \le F(x_k) - \frac{
	\rho
	- L}{2}\NM{x_{k + 1} - x_k}{2}^2$.
\end{proposition}

In nmAPG
(Algorithm~\ref{alg:nmapg}),
there is always a sufficient decrease after performing the (non-accelerated) proximal descent from $x_k$ to $x^p_{k + 1}$
(Proposition~\ref{pr:exactpx}),  
but not necessarily the case for the accelerated descent
from $x_k$ to $x^a_{k + 1}$ (which is generated by 
a possibly bad extrapolation $y_k$).
Hence, nmAPG needs to perform extra checking at step~5.
If the condition fails, $x^p_{k + 1}$ is used instead of $x^a_{k + 1}$ in step~9.


As the main problem is on $y_k$,
we propose 
to check $F(y_k)$ (step~5
in Algorithm~\ref{alg:ours})
{\bf before} the proximal step (step~11),
instead of checking after the proximal steps.
Though this change is simple, the main difficulty is how to guarantee convergence while
simultaneously maintaining acceleration and using only one proximal step.  As will be seen in Section~\ref{sec:extprox},
existing proofs do not hold even with 
exact
proximal steps.


The following shows that a similar sufficient decrease condition can still be guaranteed
after this modification.
\begin{proposition} \label{pr:suffdesc}
With exact proximal steps
in Algorithm~\ref{alg:ours}, 
$F(x_{k + 1}) \le 
\min\left( F(y_k), \Delta_k \right) 
- \frac{\rho - L}{2} \NM{x_{k + 1} - v_k}{2}^2$.
\end{proposition}

In 
step~4,
setting $q=0$ 
is the most straightforward. The $F(y_k)$ value is then checked w.r.t. the
most recent 
$F(x_k)$. 
The use of a larger $q$ is inspired from the Barzilai-Borwein scheme for unconstrained
smooth minimization
\cite{grippo2002nonmonotone}.
This allows 
$y_k$ to occasionally increase the objective, while ensuring 
$F(y_k)$ to be smaller than the
largest objective value from the last
$q$ iterations.
In the experiments, $q$ is set to $5$ as in \cite{wright2009sparse,gong2013gist}.


\subsection{Inexact Proximal Step}
\label{sec:iextpx}

Proposition~\ref{pr:suffdesc} 
requires exact proximal step,
which can be expensive.
Inexact proximal steps are much cheaper, but
the inexactness has to be carefully controlled 
to ensure convergence.
We  will propose two such schemes  
depending on whether $g$ is convex.
Note that $f$ is not required to be convex.



\noindent
\textbf{Convex $g$.}
As $g$ is convex, the optimization problem associated with the proximal step is also convex. 
Let $h_{\eta g}(x) \equiv \frac{1}{2}\NM{x - z_k}{2}^2 + \eta g(x)$ be the objective in the proximal step.
For any $z_k$, the dual of the proximal step at $z_k$
can be obtained as
\begin{equation} \label{eq:dual} 
\max_{w}
\mathcal{D}_{\eta g}(w) \equiv 
\eta \left(z_k^{\top} w - g^{*}(w)\right) - \frac{\eta^2}{2}\NM{w}{2}^2, 
\end{equation} 
where $g^*$ is the 
convex conjugate
of $g$.
In an inexact proximal step,
the obtained dual variable
$\tilde{w}_k$ 
only
approximately maximizes
(\ref{eq:dual}).
The duality gap
$\varepsilon_k \equiv h_{\eta g}(x_{k + 1}) - \mathcal{D}_{\eta g}(\tilde{w}_k)$,
where $x_{k + 1} = z_k - \eta \tilde{w}_k$,
upper-bounds the
approximation 
error  
of the inexact proximal step
$\epsilon_k \equiv h_{\eta g}(x_{k + 1}) - h_{\eta g}\left( \Px{\eta g}{z_k} \right)$.
To ensure  the inexactness $\epsilon_k$ to be smaller
than a given threshold $\tau_k$, 
we can control the duality gap as 
$\varepsilon_k\leq\tau_k$ 
\cite{schmidt2011convergence}.

The following shows that
$x_{k + 1}$ satisfies a similar sufficient decrease condition as in 
Proposition~\ref{pr:suffdesc}.
Note that this cannot be derived from \cite{schmidt2011convergence},
which relies on the 
convexity of 
$f$.

\begin{proposition} \label{pr:funcvale}
$ \!\! F(x_{k + 1}) \! \le \! 
F(v_k) - \frac{\rho - L}{2} \NM{x_{k + 1} \! - \! v_k}{2}^2+ \rho \varepsilon_k$.
\end{proposition} 



\noindent
\textbf{Nonconvex $g$.}
When $g$
is nonconvex,
the GD algorithm \cite{attouch2013convergence} allows 
inexact proximal steps. However,  it does not support acceleration,
and nonmonotone update.
Thus, its convergence proof cannot be used here.


As $g$ is nonconvex, 
it is difficult to derive
the dual of the corresponding proximal step, and 
the optimal duality gap may also be nonzero.
Thus, we monitor the progress of $F$ instead.
Inspired by Proposition~\ref{pr:exactpx},
we require 
$x_{k + 1}$ 
from an inexact proximal step to satisfy the following weaker condition:
\begin{align}
F(x_{k + 1}) \le F(v_k) - \frac{\delta}{2}\NM{x_{k + 1} - v_k}{2}^2,
\label{eq:sheother}
\end{align}
where 
$\delta \in (0, \rho - L)$. 
This condition has also been used in the 
GD algorithm.
However,
it requires checking an extra condition which is impractical.\footnote{Specifically, the condition is: $\exists$ 
	$w_{k + 1} \in \partial g(x_{k + 1})$ 
	such that $\NM{w_{k + 1} + \nabla f(v_k)}{2}^2 \le b \NM{x_{k + 1} - v_k}{2}^2$ for some constant $b > 0$.
	However, the subdifferential $\partial g(x_k)$ is in general difficult to compute.}




We could have also used condition \eqref{eq:sheother}
when $g$ is convex.
However, 
Proposition~\ref{pr:funcvale} offers more precise control,
as it can recover \eqref{eq:sheother} by setting 
$\varepsilon_k = \frac{\rho - L - \delta}{2 \rho}\NM{x_{k + 1} - v_k}{2}^2$
(note that $\delta < \rho - L$).
Besides, 
the duality gap $\varepsilon_k$ is 
readily produced by primal-dual algorithms, and
is often less expensive to compute than $F$.


\section{Convergence Analysis}
\label{sec:conv}

\begin{definition}[\cite{attouch2013convergence}] \label{def:subdifferential}
The {\em Frechet subdifferential} of $F$ at $x$ is
\begin{align*}
\hat{\partial} F(x) = \left\{ u : 
\lim_{y \neq x}\inf_{y \rightarrow x}
\frac{F(y) - F(x) - u^{\top} (y - x)}{\NM{y - x}{2}} \ge 0\right\}.
	\end{align*}
The {\em limiting subdifferential} (or simply {\em subdifferential}) of $F$ at $x$
is
	$\partial F(x) = \{ u : \exists x_k \rightarrow x, F(x_k) \rightarrow F(x), 
	u_k \in \hat{\partial} F(x_k) \!\rightarrow\! u, \text{ as } k \rightarrow \infty \}$.
\end{definition}


\begin{definition}[\cite{attouch2013convergence}]
	$x$ is a {\em critical point} of $F$ if $0 \in \nabla f(x) + \partial g(x)$.
\end{definition}

\subsection{Exact Proximal Step}
\label{sec:extprox}


In this section,
we will 
show that Algorithm~\ref{alg:ours}  (where both
$f$ and $g$ can be nonconvex)
converges with a $O(1/K)$ rate.
	This is the best known rate for nonconvex
	problems
	with first-order methods \cite{nesterov2004introductory}.
A similar $O(1/K)$ rate 
	for $\NM{\Gm{v_k}}{2}^2$ 
	is recently established 
	for APG with nonconvex $f$ but only convex $g$
	\cite{ghadimi2015accelerated}.
Note also that no convergence rate has been proved for nmAPG \cite{li2015accelerated} and GD \cite{attouch2013convergence} in this case.
Besides, their proof techniques cannot be used here as their nonmonotone updates are different.

\begin{theorem} \label{thm:exact}
The sequence $\{x_k\}$ generated from Algorithm~\ref{alg:ours}
(with exact proximal step)
have at least one limit point,
and all limit points are critical points of \eqref{eq:compfunc}.
\end{theorem}

Let  $\Gm{v} = v - \Px{\eta g}{v - \eta  \nabla f(v)}$, 
the proximal mapping at $v$ \cite{parikh2014proximal}.  
The following Lemma suggests that
$\NM{\Gm{v}}{2}^2$ 
can be used to measure how far 
$v$ is from optimality \cite{ghadimi2015accelerated}.

\begin{lemma}[\cite{gong2013gist,attouch2013convergence}] \label{lem:prox}
$v$ is a critical point of \eqref{eq:compfunc} if and only if $\Gm{v} = 0$.
\end{lemma}

The following 
Proposition shows that the proposed Algorithm~\ref{alg:ours} 
converges with a $O(1/K)$ rate.

\begin{proposition} \label{pr:exact:rate}
Let $\phi(k) = \arg\min_{t = \max(k - q, 1), \dots, k}$ $\NM{x_{t + 1} - v_t}{2}^2$.
(i) $\lim_{k \rightarrow \infty} \NM{ \Gm{v_{\phi(k)} }}{2}^2 = 0$; and (ii) $\min_{k =
1, \dots, K} \NM{\Gm{ v_{{\phi(k)}} }}{2}^2 \le \frac{2 (q + 1) c_1}{(\rho - L) K}$,
where
$c_1 = \max_{t = 1, \dots, q + 1}F(x_t) - \inf F$. 
\end{proposition}

\begin{table*}[ht]
\centering
\small
\vspace{-5px}
\caption{Results on the image inpainting experiment (CPU time is in seconds). }
\label{tab:tvimg}
\begin{tabular}{cc | c| c|  c | c|  c | c}
	\hline
	&              &         \multicolumn{2}{c|}{$\lambda=0.01$}         &
	\multicolumn{2}{c|}{$\lambda=0.02$}         &
	\multicolumn{2}{c}{$\lambda=0.04$}         \\ 
	&              & RMSE                       & CPU time               & RMSE                       & CPU time              & RMSE                       & CPU time              \\ \hline
	\multirow{4}{*}{(nonconvex)} &    GDPAN     & 0.0326$\pm$0.0001          & 212.1$\pm$50.9         & 0.0301$\pm$0.0001          & 172.6$\pm$28.4        & 0.0337$\pm$0.0001          & 151.6$\pm$57.0        \\ \cline{2-8}
	&    nmAPG     & \textbf{0.0323$\pm$0.0001} & 600.5$\pm$35.8         & \textbf{0.0299$\pm$0.0001} & 461.7$\pm$33.3        & \textbf{0.0335$\pm$0.0001} & 535.6$\pm$29.7        \\ \cline{2-8}
	& niAPG(exact) & \textbf{0.0323$\pm$0.0001} & 307.4$\pm$26.8         & \textbf{0.0299$\pm$0.0001} & 297.2$\pm$35.3        & \textbf{0.0335$\pm$0.0001} & 282.7$\pm$19.3        \\ \cline{2-8}
	&    niAPG     & \textbf{0.0323$\pm$0.0002} & \textbf{91.6$\pm$10.8} & \textbf{0.0299$\pm$0.0001} & \textbf{77.1$\pm$7.4} & \textbf{0.0335$\pm$0.0001} & \textbf{56.5$\pm$9.4} \\ \hline
	(convex)           &     ADMM     & 0.0377$\pm$0.0001          & 55.7$\pm$5.1           & 0.0337$\pm$0.0001          & 54.7$\pm$1.4          & 0.0362$\pm$0.0001          & 33.2$\pm$1.5          \\ \hline
\end{tabular}
\end{table*}

\begin{table*}[ht]
\centering
\small
\vspace{-15px}
\caption{Matrix completion performance on the synthetic data (CPU time in seconds).  Here, NMSE
	is scaled by $\times 10^{-2}$. Group (I) is based on convex nuclear norm regularization; group (II) on
	factorization model; and group (III) on nonconvex model \eqref{eq:promc}.}
\begin{tabular}{cc|ccc|ccc|ccc} \hline
	&&  \multicolumn{3}{c|}{$m=500$ (observed: $12.43\%$)}   &  \multicolumn{3}{c|}{$m=1000$
		(observed: $6.91\%$)}   &   \multicolumn{3}{c}{$m=2000$ (observed: $3.80\%$)}   \\
	&		&          NMSE          & rank &     CPU time       &          NMSE          & rank &      CPU time       &          NMSE          & rank &      CPU time       \\ \hline
	\multirow{2}{*}{(I)} &		active    &     4.10$\pm$0.16      &  42  &     11.8$\pm$1.1
	&     4.08$\pm$0.11      &  55  &     77.6$\pm$8.4     &     3.92$\pm$0.04      &  71  &
	507.3$\pm$25.4    \\ \cline{2-11}
	&ALT-Impute  &     3.99$\pm$0.15      &  42  &     1.9$\pm$0.2      &     3.87$\pm$0.09      &  55  &     29.4$\pm$1.2     &     3.68$\pm$0.03      &  71  &    143.1$\pm$3.9     \\ \hline
	\multirow{2}{*}{(II)} &AltGrad    &     2.99$\pm$0.45      &  5   &     0.2$\pm$0.1      &
	2.73$\pm$0.21      &  5   & \textbf{0.4$\pm$0.1} &     2.67$\pm$0.27      &  5   &
	\textbf{1.2$\pm$0.2} \\ \cline{2-11}
	&		R1MP     &     23.04$\pm$1.27     &  45  &     0.3$\pm$0.1      &     21.39$\pm$0.94     &  54  &     0.9$\pm$0.1      &     20.11$\pm$0.28     &  71  &     2.7$\pm$0.2      \\ \hline
	\multirow{5}{*}{(III)} &IRNN     & \textbf{1.96$\pm$0.05} &  5   &     19.2$\pm$1.2     &
	\textbf{1.88$\pm$0.04} &  5   &    215.1$\pm$4.3     & \textbf{1.80$\pm$0.03} &  5   &
	3009.5$\pm$35.9    \\ \cline{2-11}
	& FaNCL     & \textbf{1.96$\pm$0.05} &  5   &     0.4$\pm$0.1      & \textbf{1.88$\pm$0.04}
	&  5   &     1.4$\pm$0.1      & \textbf{1.80$\pm$0.03} &  5   &     5.6$\pm$0.2      \\
	\cline{2-11}
	&		nmAPG     & \textbf{1.96$\pm$0.05} &  5   &     2.3$\pm$0.2      &
	\textbf{1.88$\pm$0.03} &  5   &     6.9$\pm$0.3      & \textbf{1.80$\pm$0.03} &  5   &
	27.1$\pm$4.0     \\ \cline{2-11}
	&		niAPG(exact) & \textbf{1.96$\pm$0.04} &  5   &     1.8$\pm$0.2      &
	\textbf{1.88$\pm$0.03} &  5   &     5.3$\pm$0.5      & \textbf{1.80$\pm$0.04} &  5   &
	18.4$\pm$2.2     \\ \cline{2-11}
	&		niAPG     & \textbf{1.96$\pm$0.05} &  5   & \textbf{0.1$\pm$0.1} & \textbf{1.88$\pm$0.03} &  5   & \textbf{0.4$\pm$0.1} & \textbf{1.80$\pm$0.04} &  5   & \textbf{1.2$\pm$0.2} \\ \hline
\end{tabular}
\label{tab:sythmatcomp}
\vspace{-8px}
\end{table*}


\subsection{Inexact Proximal Step}
\label{sec:cvx_conv}

\noindent
\textbf{Convex $g$.}
As in \cite{schmidt2011convergence},
we assume that the duality gap $\varepsilon_k$  decays as $O(1/k^{1 + \varsigma})$ for some $\varsigma > 0$.
Let $c \equiv \sum_{k = 1}^{\infty} \varepsilon_k$. 
Note that $c < \infty$.

\begin{theorem} \label{thm:conv}
The sequence $\{x_k\}$ generated from Algorithm~\ref{alg:ours}
have at least one limit point,
and all limit points 
are critical points of \eqref{eq:compfunc}.
\end{theorem}

\begin{proposition} \label{pr:inextpx}
Let $e_k \equiv x_{k + 1} - \Px{\eta g}{x_k - \eta \nabla f(x_k)}$, the difference between the inexact and exact proximal step solutions at iteration $k$. 
We have $\NM{e_k}{2}^2 \le 2 \varepsilon_k$.
\end{proposition}

Note that 
the 
proof techniques in \cite{schmidt2011convergence} cannot be used here as
$f$ is not required to be convex.
As in Proposition~\ref{pr:exact:rate},
we also use $\NM{\Gm{v_{\phi(k)}}}{2}^2$ to measure  how far $v_{\phi(k)}$ is from optimality.

\begin{proposition} \label{pr:ietcvx}
(i) $\lim_{k \rightarrow \infty} \NM{\Gm{v_{\phi(k)}}}{2}^2 = 0$; and 
(ii) $\min_{k = 1, \dots, K} \NM{\Gm{v_{\phi(k)}}}{2}^2 \le \frac{2}{K} (  4 c + \frac{(q + 1)(c_1 + \rho c)}{\rho - L} )$.
\end{proposition}

When all $\varepsilon_k$'s are zero,
Proposition~\ref{pr:ietcvx} reduces to Proposition~\ref{pr:exact:rate}.
In general,
the bound of
$\min_{k = 1, \dots, K} \NM{\Gm{v_{\phi(k)}}}{2}^2$ in 
Proposition~\ref{pr:ietcvx} is larger due to the inexact proximal step.


\noindent
\textbf{Nonconvex $g$.}
With inexact proximal steps, 
nmAPG no longer guarantees convergence, and its proof cannot be easily extended.
On the other hand, GD allows inexact proximal steps
but uses a different approach to control inexactness.   Moreover, it
does not support acceleration. 

The following shows that Algorithm~\ref{alg:ours} generates a bounded sequence, and
Corollary~\ref{cor:temp1} shows that the limit points are critical points.

\begin{theorem} \label{the:temp1}
The sequence $\{x_k\}$ generated from Algorithm~\ref{alg:ours}
has at least one limit point.
\end{theorem}

\begin{corollary} \label{cor:temp1}
Let $\{ x_{k_j} \}$ be a subsequence of $\{ x_k \}$ with $
\lim_{k_j \rightarrow \infty} x_{k_j} = x_* $.
If (i) $x_{k + 1} \neq v_k$ unless $v_k = \Px{\eta g}{v_k - \eta \nabla f(v_k)}$, and (ii) $\lim_{k_j \rightarrow \infty} F(x_{k_j}) = F(x_*)$,
then $x_*$ is a critical point of \eqref{eq:compfunc}.
\end{corollary}

Assumption (i),
together with Lemma~\ref{lem:prox},
ensures
that the sufficient decrease condition in \eqref{eq:sheother} will not be 
trivially satisfied  by
$x_{k + 1} = v_k$, unless $v_k$ is a critical point.
Assumption (ii)
 follows from Definition~\ref{def:subdifferential}, as
the subdifferential is defined by a limiting process.

\begin{proposition} \cite{attouch2013convergence} \label{pr:assii}
Assumption (ii) is satisfied when (i) the proximal step is exact;
or (ii) $g$ is continuous or is the indicator function of a compact set.
\end{proposition}

When the proximal step is exact or when $g$ is convex,
$\Gm{\cdot}$ has been used to measure the distance from optimality.
However, 
this is inappropriate 
when $g$ is nonconvex and the proximal step is inexact,
as the inexactness can no longer be directly controlled.
Instead,
we will measure optimality via $a_k  \equiv \NM{x_{k + 1} - v_k}{2}^2$.

\begin{proposition} \label{pr:ietncvx2}
(i) $\lim_{k \rightarrow \infty} a_k = 0$; and
(ii) $\min_{k = 1, \dots, K} a_{\phi(k)} \le \frac{2 (q + 1) c_1}{\delta K}$.
\end{proposition}

When the proximal step is exact,
$x_{k + 1} = \Px{\eta g}{v_k - \eta \nabla f(v_k)}$,
and $a_k = \NM{\Gm{v_k}}{2}^2$.
Proposition~\ref{pr:ietncvx2}
then reduces to Proposition~\ref{pr:exact:rate} (but with a looser bound).

\begin{table*}[ht]
\centering
\small
\vspace{-5px}
\caption{Results on the \textit{MovieLens} data sets (CPU time in seconds). 
Here, RMSE is scaled by $10^{-1}$.
	Group (I) is based on convex nuclear norm regularization;
	group (II) on factorization model; and group (III) on nonconvex model
	\eqref{eq:promc}.}
\begin{tabular}{c c|ccc|ccc|ccc}
	\hline
	&              &      \multicolumn{3}{c|}{\textit{MovieLens-1M}}      &      \multicolumn{3}{c|}{\textit{MovieLens-10M}}       &       \multicolumn{3}{c}{\textit{MovieLens-20M}}        \\
	&              &          RMSE          & rank &       CPU time       &          RMSE          & rank &        CPU time        &          RMSE          & rank &        CPU time         \\ \hline
	\multirow{2}{*}{(I)} & active       &     8.20$\pm$0.01      &  68  &     50.5$\pm$1.6     &     8.14$\pm$0.01      & 101  &    1520.8$\pm$18.2     &     8.02$\pm$0.01      & 197  &    7841.9$\pm$666.3     \\ \cline{2-11}
	& ALT-Impute   &     8.18$\pm$0.01      &  68  &     34.0$\pm$1.1     &     8.14$\pm$0.01      & 101  &     821.7$\pm$34.5     &     8.01$\pm$0.01      & 197  &    3393.2$\pm$220.3     \\ \hline
	\multirow{2}{*}{(II)} & AltGrad      &     8.02$\pm$0.03      &  6   &     4.0$\pm$1.1      &     7.97$\pm$0.04      &  9   &     94.5$\pm$30.8      &     7.94$\pm$0.04      &  10  &     298.3$\pm$54.1      \\ \cline{2-11}
	& R1MP         &     8.53$\pm$0.02      &  13  & \textbf{1.3$\pm$0.2} &     8.52$\pm$0.04      &  23  & \textbf{58.8$\pm$11.0} &     8.54$\pm$0.02      &  26  & \textbf{139.2$\pm$23.7} \\ \hline
	\multirow{4}{*}{(III)} & FaNCL        & \textbf{7.88$\pm$0.01} &  5   &     12.5$\pm$0.9     & \textbf{7.79$\pm$0.01} &  8   &     703.5$\pm$18.3     & \textbf{7.84$\pm$0.03} &  9   &    2296.9$\pm$176.4     \\ \cline{2-11}
	& nmAPG        & \textbf{7.87$\pm$0.01} &  5   &     12.5$\pm$0.9     & \textbf{7.80$\pm$0.01} &  8   &     627.5$\pm$16.4     & \textbf{7.85$\pm$0.01} &  9   &    1577.9$\pm$103.2     \\ \cline{2-11}
	& niAPG(exact) & \textbf{7.87$\pm$0.01} &  5   &     11.1$\pm$0.8     & \textbf{7.79$\pm$0.01} &  8   &     403.1$\pm$19.6     & \textbf{7.84$\pm$0.01} &  9   &     1111.9$\pm$65.3     \\ \cline{2-11}
	& niAPG        & \textbf{7.87$\pm$0.01} &  5   &     2.7$\pm$0.3      & \textbf{7.79$\pm$0.01} &  8   &      90.2$\pm$2.6      & \textbf{7.85$\pm$0.01} &  9   &     257.6$\pm$33.4      \\ \hline
\end{tabular}
\label{tab:mvlens}
\end{table*}

\begin{table*}[ht]
	\centering
	\small
	\vspace{-15px}
	\caption{Number of proximal steps on the synthetic and recommender system data sets.}
	\label{tab:callpxrec}
	\begin{tabular}{c| c |c | c| c| c| c |c | c} \hline
		&             \multicolumn{3}{c|}{\textit{Synthetic}}              &                \multicolumn{3}{c|}{\textit{MovieLens}}                 & \multirow{2}{*}{\textit{Netflix}} & \multirow{2}{*}{\textit{Yahoo}} \\
		& $m=$500               & $m=$1000              & $m=$2000               & \textit{1M}           & \textit{10M}           & \textit{20M}           &                                   &  \\ \hline
		nmAPG     & 77                    & 104                   & 145                    & 95                    & 221                    & 236                    & 183                               & 579                             \\ \hline
		niAPG(exact) & 64 ($\downarrow \!\! 17\%$) & 85 ($\downarrow \!\! 18\%$) & 115 ($\downarrow \!\! 21\%$) & 82 ($\downarrow \!\! 13\%$) & 143 ($\downarrow \!\! 35\%$) & 165 ($\downarrow \!\! 31\%$) & 133 ($\downarrow \!\! 27\%$)            & 425 ($\downarrow \!\! 26\%$)          \\ \hline
		niAPG     & 64 ($\downarrow \!\! 17\%$) & 85 ($\downarrow \!\! 18\%$) & 115 ($\downarrow \!\! 21\%$)
		& 81 ($\downarrow \!\! 15\%$) & 140 ($\downarrow \!\! 36\%$) & 160 ($\downarrow \!\! 32\%$) & 132
		($\downarrow \!\! 28\%$)            & 413 ($\downarrow \!\! 29\%$)          \\ \hline
	\end{tabular}
	\vspace{-8px}
\end{table*}

\section{Experiments}
\label{sec:expt}

In this section,
we perform experiments when
$g$ 
is convex 
(Section~\ref{sec:tvmdl}) and nonconvex (Section~\ref{sec:colfilter}).

\subsection{Image Inpainting}
\label{sec:tvmdl}

The total variation (TV) model \cite{beck2009tv} has been popularly used in image processing.
Let $y \in \R^d$ be the vectorized input image and $x \in \R^d$ be the recovered one.
We consider the TV model with nonconvex log-sum-penalty  regularizer
\cite{candes2008enhancing}.
\begin{equation}
\min_x 
\frac{1}{2}\NM{M \odot (x - y)}{2}^2 
+   \lambda \sum_{i = 1}^d 
\kappa([ D_h x ]_i) + \kappa([ D_v x ]_i),
\label{eq:imgtv}
\end{equation}
where
$M \in \{0, 1\}^d$ is a mask such that $M_{ij}=1$ indicates that the corresponding pixel is observed, 
$D_h$ and $D_v$ are the horizontal and vertical partial derivative operators,
$\odot$ is the elementwise multiplication, and
$\kappa(\alpha) = \log(1 + |\alpha|)$.

As suggested in \cite{qyao2016icml}, 
\eqref{eq:imgtv}
can be transformed as the minimization of 
$f(x) + \lambda \TV{x}$,
where 
$ f(x) 
= 
\frac{1}{2}\NM{M \odot (x - y)}{2}^2 
- \lambda [ \TV{x}
+ \sum_{i = 1}^d \kappa([ D_h x ]_i) + \kappa([ D_v x ]_i) ]$
is nonconvex but smooth,
and $\TV{x} = \NM{D_h x}{1} + \NM{D_v x}{1}$ is the standard (convex) TV regularizer.
Thus,
we only need to handle the proximal step of the TV regularizer,
which will be computed numerically by solving
its dual
using L-BFGS 
\cite{beck2009tv}.

The following solvers on the transformed problems are compared:
(i) GDPAN \cite{zhong2014gradient}, which
performs
gradient descent with the proximal average;
(ii) nmAPG;
(iii) the proposed niAPG,
in which inexactness of the proximal step is controlled by decaying the duality gap $\varepsilon_k$ at a rate of $O(1/k^{1.5})$; and
(iv) the exact version of niAPG(exact), which simulates an exact proximal step with a small duality gap (${10}^{-4}$).
We do not compare
with the GD algorithm \cite{attouch2013convergence},
as its inexactness condition is difficult to check and it does not use acceleration.

As a further baseline, we compare with 
the convex TV model:
$\min_{x} \frac{1}{2}\NM{M \odot (x - y)}{2}^2 + \lambda (\NM{D_v x}{1} + \NM{D_h x}{1})$,
which is solved using ADMM \cite{boyd2011distributed}. 
We do not compare with
CCCP \cite{yuille2002concave}, which
is slow in practice
\cite{qyao2016icml}.


Experiments are performed on 
the ``Lena'' image\footnote{\url{http://www.cs.tut.fi/~foi/GCF-BM3D/images/image_Lena512rgb.png}}.
We normalize
the pixel values 
to $[0,1]$, and then
add Gaussian noise from $\mathcal{N}(0, 0.05)$.
$50\%$ of the pixels
are randomly sampled as observed.
For performance evaluation,
we report the 
CPU time
and 
root-mean-squared error
(RMSE) 
on the whole image.
The experiment is repeated five times.
Results are shown in 
Table~\ref{tab:tvimg}.\footnote{In all the tables, 
the boldface 
indicates the best and comparable results (according to the pairwise t-test with 95\% confidence).}
Figure~\ref{fig:imgtv} plots convergence of the objective.\footnote{
	Because of the lack of space,
	the plot for $\lambda = 0.01$
	is not shown.}
As can be seen, the nonconvex TV model has better RMSE than
the convex one. Among the nonconvex models,
niAPG is much faster, as
it only requires a small number
of cheap inexact proximal steps 
(Table~\ref{tab:callprox1}).

\begin{figure}[ht]
\centering
\subfigure[$\lambda = 0.02$.]
{\includegraphics[height = 0.18\textwidth]{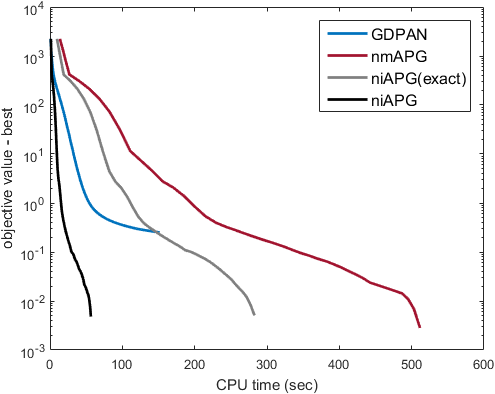}}\quad
\subfigure[$\lambda = 0.04$.]
{\includegraphics[height = 0.18\textwidth]{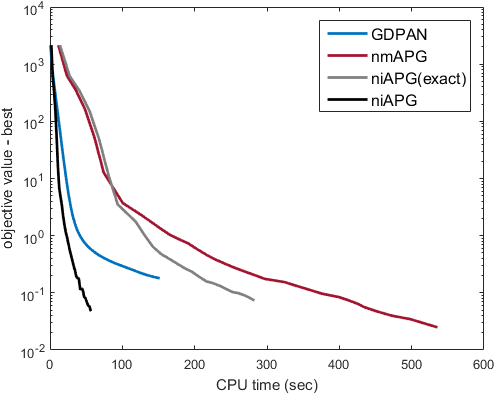}}
\vspace{-10px}
\caption{Objective value vs CPU time on the image data.}
\label{fig:imgtv}
\vspace{-10px}
\end{figure}

\begin{table}[ht]
\centering
\small
\vspace{-10px}
\caption{Number of proximal steps on image data. Number in brackets is the percentage reduction 
w.r.t. nmAPG.}
\label{tab:callprox1}
\begin{tabular}{c| c| c| c}
	\hline
	             & $\lambda = 0.01$      & $\lambda = 0.02$      & $\lambda = 0.04$      \\ \hline
	   nmAPG     & 87                    & 46                    & 43                    \\ \hline
	niAPG(exact) & 47 ($\downarrow \!\! 46\%$) & 35 ($\downarrow \!\! 24\%$) & 29 ($\downarrow \!\! 33\%$) \\ \hline
	   niAPG     & 57 ($\downarrow \!\! 35\%$) & 41 ($\downarrow \!\! 11\%$) & 28 ($\downarrow \!\! 35\%$) \\ \hline
\end{tabular}
\vspace{-10px}
\end{table}

\begin{figure*}[ht]
	\centering
	\subfigure[\textit{MovieLens-20M}.]
	{\includegraphics[height = 0.18\textwidth]{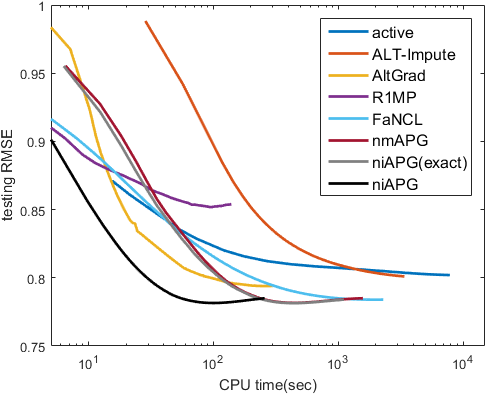} 
		\label{fig:netflix:20M}} 
	\qquad\qquad
	\subfigure[\textit{Netflix}.]
	{\includegraphics[height = 0.18\textwidth]{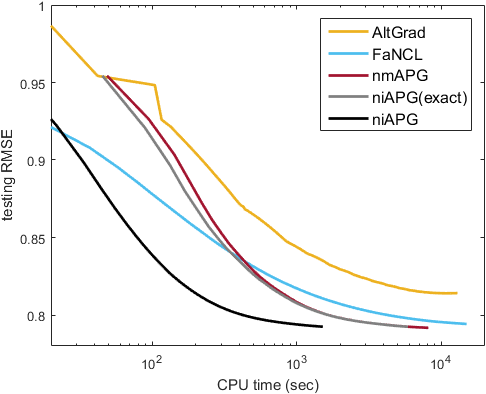}  
		\label{fig:netflix:netflix}} 
	\qquad\qquad
	\subfigure[\textit{Yahoo}.]
	{\includegraphics[height = 0.18\textwidth]{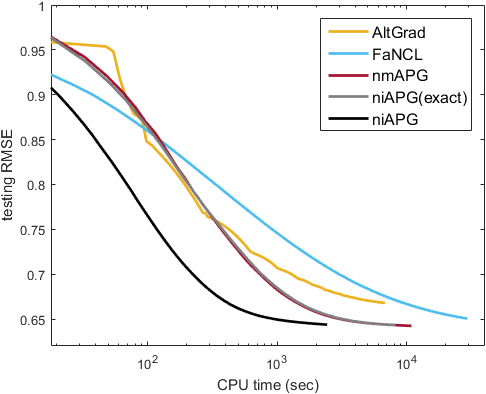}
		\label{fig:netflix:yahoo}}
	\vspace{-10px}
	\caption{Testing RMSE vs CPU time on the recommendation system data sets.}
	\label{fig:netflix}
	\vspace{-10px}
\end{figure*}

\subsection{Matrix Completion}
\label{sec:colfilter}

In this section, we consider matrix completion with a nonconvex low-rank regularizer.
As shown in \cite{canyi2014,yao2015fast},
it gives better performance than 
nuclear-norm based 
and 
factorization 
approaches.
The optimization problem can be formulated as
\begin{align}
\min_{\rank(X) \le r}  
\frac{1}{2} \NM{\SO{X_{ij}- O_{ij}}}{F}^2
+ \lambda \sum_{i = 1}^r \kappa\left( \sigma_i(X) \right),
\label{eq:promc}
\end{align}
where 
$O_{ij}$s are the observations,
$\Omega_{ij} = 1$ 
if 
$O_{i,j}$ is 
observed, and 0 otherwise,
$\sigma_i(X)$ is the $i$th leading singular value of $X$,
and $r$ is the desired rank.
The associated proximal step  
can be solved with rank-$r$ SVD  
\cite{canyi2014}.

\noindent
\textbf{Synthetic Data.}
The observed $m\times m$ matrix is generated as $O = 
U V + G$,  where 
the entries
of $U \in \R^{m \times k}, V \in \R^{k \times m}$ 
(with $k = 5$) are 
sampled i.i.d. from the normal distribution $\mathcal{N}(0,
1)$, and
entries of $G$ sampled from $\mathcal{N}(0, 0.1)$.
A total of $\|\Omega\|_1 = 2 m k \log(m)$ random entries
in $O$ are observed.  
Half of them are used for training, and the rest as validation
set.

In
the proposed
niAPG algorithm,
its proximal step is approximated by
using power method \cite{halko2011finding}, and inexactness of the proximal step is
monitored by condition~\eqref{eq:sheother}.  Its variant niAPG(exact) has 
exact
proximal steps 
computed by the Lancoz algorithm
\cite{larsen1998lanczos}. 
They are 
compared with the following solvers
on the nonconvex model \eqref{eq:promc}:
	(i) Iterative reweighted nuclear norm (IRNN)
	\cite{canyi2014} algorithm;
	(ii) Fast nonconvex low-rank learning (FaNCL) algorithm \cite{yao2015fast},
	using the power method to approximate the proximal step;
	(iii) nmAPG,
	in which the proximal step is exactly computed by the Lancoz algorithm.
We also compare with other matrix completion algorithms, including the well-known
(convex) nuclear-norm-regularized algorithms:
(i) active subspace selection  
\cite{hsieh2014nuclear}
and
(ii) ALT-Impute \cite{hastie2015matrix}.
We also compare with
state-of-the-art 
factorization models
(where the
rank 
is tuned by the validation set):
(i) R1MP \cite{wang2015rankone};
and (ii) state-of-the-art gradient descent
based AltGrad \cite{zhao2015nonconvex}.
We do not compare with the Frank-Wolfe algorithm \cite{zhang2012accelerated},
which has been shown to be slower
\cite{hsieh2014nuclear}.

Testing is performed on the non-observed 
entries (denoted
$\bar{\Omega}$).
Three measures are used
for performance evaluation:
(i) normalized mean squared error
(NMSE): $\sqrt{\NM{P_{\bar{\Omega}} (X - UV)}{F}^2} / \sqrt{\NM{P_{\bar{\Omega}} (UV) }{F}^2}$;
(ii) rank of $X$;
and 
(iii) training time.
Each experiment is repeated five times.

Table~\ref{tab:sythmatcomp}
shows
the performance.
Convergence for algorithms solving \eqref{eq:promc} is shown\footnote{Because of the lack of space, the plot for $m = 500$ is not shown.}
in Figure~\ref{fig:sythmatcomp}.
As has also been 
observed in \cite{canyi2014,yao2015fast},
nonconvex regularization yields lower NMSE than 
nuclear-norm-regularized 
and factorization
models.
Again, niAPG is the fastest. Its speed is comparable with AltGrad, but is more accurate.
Table~\ref{tab:callpxrec}
compares 
the numbers
of proximal steps.
As can be seen, 
both nmAPG(exact) and nmAPG require significantly fewer proximal steps.

\noindent
\textbf{Recommender Systems.}
We first consider the 
\textit{MovieLens} data sets
(Table~\ref{tab:recsys:dataset}),
which contain ratings of different users on movies or musics.
We follow the setup in \cite{wang2015rankone,yao2015fast},  
and use $50\%$ of the observed ratings for training, $25\%$ for validation and the rest for testing.
For performance evaluation, we use the root mean squared error on the test set $\bar{\Omega}$:
$\text{RMSE} = \sqrt{\NM{P_{\bar{\Omega}} (O - X)}{F}^2}/\sqrt{\|\bar{\Omega}\|_1}$,
rank of the recovered matrix $X$, and CPU time.
The experiment is repeated five times.

\begin{table}[ht]
	\centering
	\small
	\vspace{-15px}
	\caption{Recommender system data sets used.}
	\begin{tabular}{c c |c|c|c}
		\hline
		&              & \#users   & \#items &  \#ratings  \\ \hline
		\multirow{3}{*}{\textit{MovieLens}} & \textit{1M}  & 6,040     &  3,449  &   999,714   \\ \cline{2-5}
		& \textit{10M} & 69,878    & 10,677  & 10,000,054  \\ \cline{2-5}
		& \textit{20M} & 138,493   & 26,744  & 20,000,263  \\ \hline
		\multicolumn{2}{c|}{\textit{Netflix}}    & 480,189   & 17,770  & 100,480,507 \\ \hline
		\multicolumn{2}{c|}{\textit{Yahoo}}     & 1,000,990 & 624,961 & 262,810,175 \\ \hline
	\end{tabular}
	\vspace{-5px}
	\label{tab:recsys:dataset}
\end{table}

Table~\ref{tab:mvlens} shows the
recovery performance.
IRNN is not compared as it is too slow.
Again, the nonconvex model 
consistently outperforms nuclear-norm-regularized
and factorization models.
R1MP is the fastest, but its
recovery performance
is poor.
Figure~\ref{fig:netflix:20M} shows the convergence,
and Table~\ref{tab:callpxrec}
compares the numbers of proximal steps.
niAPG(exact) is faster than nmAPG due to the use of fewer proximal steps.
niAPG is even faster with the use of inexact proximal steps.

\begin{figure}[ht]
	\centering
	\subfigure[$m = 1000$.]
	{\includegraphics[height = 0.18\textwidth]{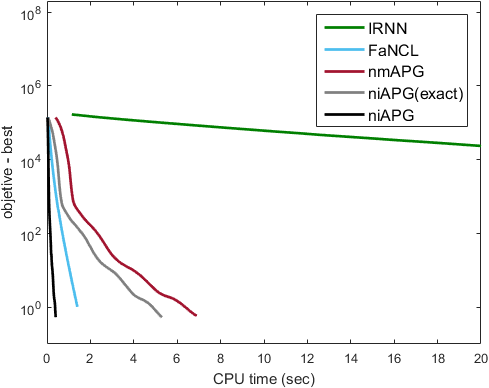}}\quad
	\subfigure[$m = 2000$.]
	{\includegraphics[height = 0.18\textwidth]{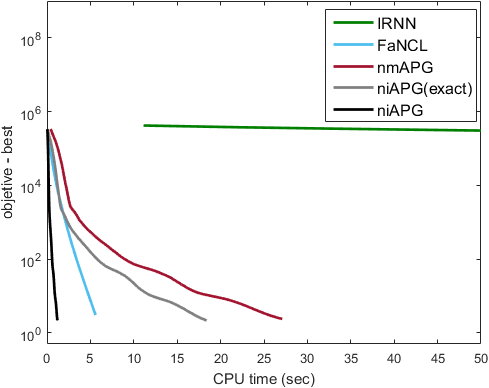}}
	\vspace{-10px}
	\caption{Objective value vs CPU time on the synthetic matrix completion data set.}
	\label{fig:sythmatcomp}
	\vspace{-5px}
\end{figure}

Finally, we perform experiments on the large 
\textit{Netflix}
and \textit{Yahoo}
data sets
(Table~\ref{tab:recsys:dataset}).
We randomly use $50\%$ 
of the observed ratings 
for training,
$25\%$ for validation and the rest for testing.
Each experiment is repeated five times. 
We do not compare with nuclear-norm-regularized methods 
as they yield higher rank and RMSE than others.
Table~\ref{tab:netflix}
shows 
the recovery performance,
and Figures~\ref{fig:netflix:netflix}, \ref{fig:netflix:yahoo}
show the convergence.
Again, niAPG is the fastest and most accurate.

\begin{table}[ht]
\centering
\small
\vspace{-10px}
\caption{Results on the \textit{Netfix} and \textit{Yahoo} data sets.
Here, RMSE is scaled by $\times 10^{-1}$.}
\begin{tabular}{c c|ccc}
	\hline
                                &              &          RMSE          & rank &    CPU time (min)     \\ \hline
	\multirow{5}{*}{\textit{Netflix}} & AltGrad      &     8.16$\pm$0.02      &  15  &     221.7$\pm$5.6     \\ \cline{2-5}
                                & FaNCL        &     7.94$\pm$0.01      &  13  &    240.8$\pm$22.7     \\ \cline{2-5}
                                & nmAPG        & \textbf{7.92$\pm$0.01} &  13  &     132.8$\pm$2.1     \\ \cline{2-5}
                                & niAPG(exact) & \textbf{7.92$\pm$0.01} &  13  &     97.7$\pm$1.8      \\ \cline{2-5}
                                & niAPG        & \textbf{7.92$\pm$0.01} &  13  & \textbf{25.2$\pm$0.6} \\ \hline
 \multirow{5}{*}{\textit{Yahoo}}  & AltGrad      &     6.69$\pm$0.01      &  14  &     112.9$\pm$4.2     \\ \cline{2-5}
                                & FaNCL        & \textbf{6.54$\pm$0.01} &  9   &    487.6$\pm$32.0     \\ \cline{2-5}
                                & nmAPG        & \textbf{6.53$\pm$0.01} &  9   &     184.3$\pm$6.3     \\ \cline{2-5}
                                & niAPG(exact) & \textbf{6.53$\pm$0.01} &  9   &     140.7$\pm$5.8     \\ \cline{2-5}
                                & niAPG        & \textbf{6.53$\pm$0.01} &  9   & \textbf{38.7$\pm$2.3} \\ \hline
\end{tabular}
\label{tab:netflix}
\vspace{-10px}
\end{table}


\section{Conclusion}

In this paper,
we proposed an efficient
accelerated proximal gradient algorithm for nonconvex problems.
Compared with the state-of-the-art \cite{li2015accelerated},
the proximal step can be inexact
and the number of proximal steps required is significantly reduced,
while still ensuring convergence to a critical point.
Experiments on 
image denoising and matrix completion problems
show that the proposed algorithm has comparable (or even better) prediction performance as the state-of-the-art,
but is much faster.


\section*{Acknowledgments}

This research project is partially funded by Microsoft Research Asia
and the Research Grants Council of the Hong Kong Special Administrative Region (Grant 614513).
The first author would thank helpful discussion and suggestions from Lu Hou and Yue Wang.

{
\small
\bibliographystyle{named}
\bibliography{bib}
}


\appendix

\section{Proof}

\subsection{Proposition~\ref{pr:suffdesc}}

From Proposition~\ref{pr:exactpx},
if the proximal step is exact at Algorithm~\ref{alg:ours}, we have 
\begin{align*}
F(x_{k + 1})
\le F(v_k) - \frac{\rho - L}{2}\NM{x_{k + 1} - v_k}{2}^2.
\end{align*}
Due to the condition at step~5,
two cases are considered. 
Let $\alpha = \rho - L$, then
\begin{itemize}
	\item If $v_k = x_k$, then $\Delta_k < F(y_k)$ and
	\begin{align}
	F(x_{k + 1}) 
	& \le F(x_k) - \frac{\alpha}{2}\NM{x_{k + 1} - x_k}{2}^2
	\notag \\
	& \le \Delta_k - \frac{\alpha}{2}\NM{x_{k + 1} - x_k}{2}^2.
	\label{eq:app1}
	\end{align}
	
	\item If $v_{k} = y_k$, then $F(y_k) \le \Delta_k$ and
	\begin{align}
	F(x_{k + 1}) 
	\le F(y_k) - \frac{\alpha}{2}\NM{x_{k + 1} - y_k}{2}^2
	\label{eq:app2}
	\end{align}
\end{itemize} 
By combining \eqref{eq:app1} and \eqref{eq:app2},
we have
\begin{align*}
F(x_{k + 1})  
\le \min(F(y_k), \Delta_k) - \frac{\alpha}{2}\NM{x_{k + 1} - y_k}{2}^2.
\end{align*}

\subsection{Proposition~\ref{pr:funcvale}}

Let $\varsigma(w) = ( w - v_k)^{\top} \nabla f(v_k) 
+ \frac{\rho}{2} \NM{w - v_k}{2}^2
+ g(w)$. 
We have
\begin{eqnarray}
\hat{x} & = & \arg\min_w h_{\frac{1}{\rho}g}(w) = \arg\min_w \varsigma(w), \label{eq:temp13} \\
\varsigma(w) & = & \rho h_{\frac{1}{\rho}g}(w) - \frac{1}{\rho}\NM{\nabla f(v_k)}{2}^2.
\label{eq:temp14}
\end{eqnarray}
From \eqref{eq:temp13}, we have
\begin{align}
\varsigma(\hat{x}) 
& = \langle  \hat{x} - v_k, \nabla f(v_k) \rangle 
+ \frac{\rho}{2} \NM{\hat{x} - v_k}{2}^2
+ g(\hat{x})
\notag \\
& \le g(v_k).
\label{eq:temp15}
\end{align}
As $h_{\frac{1}{\rho}g}(x_{k + 1}) - h_{\frac{1}{\rho}g}(\hat{x}) \le \varepsilon_k$, 
from \eqref{eq:temp14} (note that $\NM{\nabla f(v_k)}{2}^2$ is a constant), we have
\begin{align*}
\varsigma(x_{k + 1}) - \varsigma(\hat{x})
= \rho ( h(x_{k + 1}) - h(\hat{x}) )
\le \rho \epsilon
\end{align*}
With \eqref{eq:temp15},
we have
$\varsigma(x_{k + 1}) \le \rho \varepsilon_k + \varsigma(\hat{x}) \le g(v_k) + \rho \varepsilon_k$,
and then
\begin{align}
(x_{k + 1} - v_k)^{\top} \nabla f(v_k)
& + \frac{\rho}{2} \NM{x_{k + 1} - v_k}{2}^2
+ g(x_{k + 1})
\notag \\
& \le g(v_k) + \rho \varepsilon_k.
\label{eq:app4}
\end{align}
As $f$ is $L$-Lipschitz smooth,
\begin{align*}
f(x_{k + 1})
\le f(v_k) + (x_{k + 1} - v_k)^{\top} \nabla f(v_k)
+ \frac{L}{2}\NM{x_{k + 1} - v_k}{2}^2.
\end{align*}
Combining with \eqref{eq:app4},
we obtain
\begin{align*}
& f(x_{k + 1}) 
+ \frac{\rho}{2} \NM{x_{k + 1} - v_k}{2}^2
+ g(x_{k + 1})
\\
& \le 
f(v_k) + \frac{L}{2}\NM{x_{k + 1} - v_k}{2}^2
+ g(v_k) + \rho \varepsilon_k.
\end{align*}
Thus, 
$F(x_{k + 1}) \le F(v_k) - \frac{\rho - L}{2} \NM{x_{k + 1} - v_k}{2}^2 + \rho \varepsilon_k$.

\subsection{Theorem~\ref{thm:exact}}

We first introduce one proposition.

\begin{proposition} \label{pr:app1}
	Let $\phi(k) = \arg\min_{t = \max(1, k - q), \dots, k} \NM{x_{t + 1} - v_t}{2}^2$,
	then 
	$\Delta_{k + q + 1}
	\le \Delta_k
	- \frac{\rho - L}{2} \NM{x_{\phi(k) + 1} - v_{\phi(k)}}{2}^2$.
\end{proposition}

\begin{proof}
	From Proposition~\ref{pr:suffdesc},
	we have
	\begin{align*}
	F(x_{k + 1})
	& \le 
	\min\left( F(y_k), \Delta_k \right) 
	- \frac{\rho - L}{2} \NM{x_{k + 1} - v_k}{2}^2
	\\
	& \le \Delta_k 
	- \frac{\rho - L}{2} \NM{x_{k + 1} - v_k}{2}^2.
	\end{align*}
	Thus, $F(x_{k + 1}) \le \Delta_k$. Then  
	\begin{align*}
	F(x_{k + 2})
	& \le \Delta_{k + 1} 
	- \frac{\rho - L}{2} \NM{x_{k + 2} - v_{k + 1}}{2}^2
	\\
	& \le \max(F(x_{k + 1}), \Delta_k)
	- \frac{\rho - L}{2} \NM{x_{k + 2} - v_{k + 1}}{2}^2
	\\
	& \le \Delta_k
	- \frac{\rho - L}{2} \NM{x_{k + 2} - v_{k + 1}}{2}^2.
	\end{align*}
	Inducting for $t = 0, \dots, q$, we have
	\begin{align*}
	F(x_{k + 1 + t})
	\le \Delta_k
	- \frac{\rho - L}{2} \NM{x_{k + 1 + t} - v_{k + t}}{2}^2.
	\end{align*}
	Recall that $\Delta_k = \max_{t = k - q, \dots, k} F(x_t)$, thus
	\begin{align*}
	\Delta_{k + q + 1}
	& = \max_{t = k + 1, \dots, k + 1 + q}
	F(x_{t})
	\\
	& \le \Delta_k
	- \min_{t = k, \dots, k + 1}\frac{\rho - L}{2} \NM{x_{t + 1} - v_{t}}{2}^2.
	\end{align*}
	Let $\phi(k) = \arg\min_{t = k, \dots, k + q} \NM{x_{t + 1} - v_t}{2}^2$,
	and then we get the proposition. 
\end{proof}

Summing over $k = 1, \dots, K$, then
\begin{align}
\sum_{t = 0}^{q}
&
\Delta_{t + 1} - \Delta_{K + t - q + 1}
\notag \\
& \ge \frac{\rho - L}{2} \sum_{k = 1}^K \NM{x_{\phi(k) + 1} - v_{\phi(k)}}{2}^2.
\label{eq:app3}
\end{align}
As $\inf F > - \infty$,
let $K \rightarrow \infty$,
then 
\begin{align}
\sum_{t = 0}^{q}
\Delta_{t + 1} - \Delta_{K + t - q + 1}
& \le 
\sum_{t = 0}^{q}
(\Delta_{t + 1} - \inf F)
\label{eq:app5} \\
& \le (q + 1) \left( \max_{t = 1, \dots, q + 1}F(x_t) - \inf F \right) 
\notag
\end{align}
Let 
$c_1 = (q + 1) \left( \max_{t = 1, \dots, q + 1}F(x_t) - \inf F \right)$,
and it is easy to see $c_1 < \infty$ is a finite number.
Thus, from \eqref{eq:app3}, we have
\begin{align}
\lim_{k \rightarrow \infty}
\NM{x_{\phi(k) + 1} - v_{\phi(k)}}{2} = 0,
\label{eq:app10}
\end{align}
which means $\lim_{k \rightarrow \infty} x_{\phi(k) + 1} = v_{\phi(k)}$,
and both sequence $\{ x_k \}$ and $\{ v_k \}$ are bounded,
and then both have at least one limit point.

Let $\{v_{k_j}\}$ be a subsequence of $\{ v_{k_j} \}$ and $x_* = \lim_{j \rightarrow \infty} x_{k_j + 1} = \lim_{k_j \rightarrow\infty} v_{k_j}$ be a limit point,
then
\begin{align*}
\lim_{k_j \rightarrow \infty} x_{k_j + 1} 
& = \lim_{j\rightarrow\infty} v_{k_j} 
\\
& = \lim_{j\rightarrow\infty} \Px{\eta g}{v_{k_j} - \eta \nabla f(v_{k_j})} = x_*.
\end{align*}
Besides, 
as $x_{k_j + 1}$ is generated from proximal step on $v_{k_j}$, then
\begin{align*}
& (x_{k_j + 1} - v_{k_j})^{\top} \nabla f(v_{k_j})
+ \frac{\rho}{2} \NM{x_{k_j + 1} - v_{k_j}}{2}^2 + g(x_{k_j + 1})
\\
& \le 
(x_* - v_{k_j})^{\top} \nabla f(v_{k_j})
+ \frac{\rho}{2} \NM{x_* - v_{k_j}}{2}^2 + g(x_*)
\end{align*}
Thus, 
\begin{align}
\lim_{j \rightarrow \infty} \sup g(x_{k_j + 1}) \le g(x_*).
\label{eq:app14}
\end{align}
As $g$ is lower semicontinuous, we have
\begin{align}
\lim_{k_j \rightarrow \infty} g(x_{k_j + 1}) \ge g(x_*).
\label{eq:app15}
\end{align}
Combining \eqref{eq:app14} and \eqref{eq:app15},
we have
\begin{align*}
\lim_{k_j \rightarrow \infty} F(x_{k_j + 1}) = F(x_*),
\end{align*}
and then with Lemma~\ref{lem:prox},
$x_*$ is a critical point of \eqref{eq:compfunc}.

\subsection{Proposition~\ref{pr:exact:rate}}

Note that $\Gm{ v_{\phi(k)} } = x_{\phi(k) + 1} - v_{\phi(k)}$.
Using \eqref{eq:app10},
we have
\begin{align*}
\lim_{k \rightarrow \infty} \NM{ \Gm{v_{\phi(k)}} }{2}^2 = 0.
\end{align*}
Then,
using \eqref{eq:app3} and \eqref{eq:app5}, we have
\begin{align*}
\sum_{k = 1}^K \NM{x_{\phi(k) + 1} - v_{\phi(k)}}{2}^2
\le \frac{2 c_1}{\rho - L}
\end{align*}
Thus,
\begin{align*}
& \min_{k = 1, \dots, K} \NM{\Gm{ v_{\phi(k)} }}{2}^2 
\\
& = \min_{k = 1, \dots, K} \NM{x_{\phi(k) + 1} - v_{\phi(k)}}{2}^2
\\
& \le
\frac{1}{K} \sum_{k = 1}^K \NM{x_{\phi(k) + 1} - v_{\phi(k)}}{2}^2
\le 
\frac{2 c_1}{(\rho - L) K}.
\end{align*}

\subsection{Theorem~\ref{thm:conv}}

We first extend Proposition~\ref{pr:app1} to Proposition~\ref{pr:app2} here.

\begin{proposition} \label{pr:app2}
	$\Delta_{k + q + 1}
	\le \Delta_k
	- \frac{\rho - L}{2} \NM{x_{\phi(k) + 1} - v_{\phi(k)}}{2}^2
	+ \rho \sum_{t = k}^{k + q} \varepsilon_t$.
\end{proposition}

\begin{proof}
	From Proposition~\ref{pr:funcvale},
	we have
	\begin{align*}
	F(x_{k + 1})
	& \le 
	\min\left( F(y_k), \Delta_k \right) 
	- \frac{\rho - L}{2} \NM{x_{k + 1} - v_k}{2}^2 + \rho \varepsilon_k
	\\
	& \le \Delta_k 
	- \frac{\rho - L}{2} \NM{x_{k + 1} - v_k}{2}^2
	+ \rho \varepsilon_k.
	\end{align*}
	Thus, $F(x_{k + 1}) \le \Delta_k + \rho \varepsilon_k$. Then  
	\begin{align*}
	& F(x_{k + 2})
	\\
	& \le \Delta_{k + 1} 
	- \frac{\rho - L}{2} \NM{x_{k + 2} - v_{k + 1}}{2}^2
	+ \rho \varepsilon_{k + 1}
	\\
	& \le \max(F(x_{k + 1}), \Delta_k)
	- \frac{\rho - L}{2} \NM{x_{k + 2} - v_{k + 1}}{2}^2
	+ \rho \varepsilon_{k + 1}
	\\
	& \le \Delta_k
	- \frac{\rho - L}{2} \NM{x_{k + 2} - v_{k + 1}}{2}^2
	+\sum_{t = k}^{k + 1}\varepsilon_t.
	\end{align*}
	Inducting for $t = 0, \dots, q$, we have
	\begin{align*}
	F(x_{k + 1 + t})
	\le \Delta_k
	- \frac{\rho - L}{2} \NM{x_{k + 1 + t} - v_{k + t}}{2}^2
	+ \rho \sum_{t = k}^{k + q}\varepsilon_t.
	\end{align*}
	Recall that $\Delta_k = \max_{t = k - q, \dots, k} F(x_t)$, then
	\begin{align*}
	\Delta_{k + q + 1}
	& = \max_{t = k + 1, \dots, k + 1 + q}
	F(x_{t})
	\\
	& \le \Delta_k
	- \min_{t = k, \dots, k + 1}\frac{\rho - L}{2} \NM{x_{t + 1} - v_{t}}{2}^2
	+ \rho \sum_{t = k}^{k + q}\varepsilon_t.
	\end{align*}
	Thus,
	we get the proposition.
\end{proof}

Using Proposition~\ref{pr:app2},
sum over $k = 1, \dots, K$, we have
\begin{align}
& \sum_{t = 0}^{q}
\Delta_{t + 1} - \Delta_{K + t - q + 1}
\notag \\
& \ge \frac{\rho - L}{2} \sum_{k = 1}^K \NM{x_{\phi(k) + 1} - v_{\phi(k)}}{2}^2
- \rho \sum_{t = k}^{\min(k + q, K)} \varepsilon_t
\notag \\
& \ge \frac{\rho - L}{2} \sum_{k = 1}^K \NM{x_{\phi(k) + 1} - v_{\phi(k)}}{2}^2
- \rho \sum_{t = k}^{k + q} \varepsilon_t
\notag \\
& \ge \frac{\rho - L}{2} \sum_{k = 1}^K \NM{x_{\phi(k) + 1} - v_{\phi(k)}}{2}^2
-  q \rho \sum_{k = 1}^{K} \varepsilon_k
\notag \\
& \ge \frac{\rho - L}{2} \sum_{k = 1}^K \NM{x_{\phi(k) + 1} - v_{\phi(k)}}{2}^2
- q \rho c.
\label{eq:app6}
\end{align}
where the last inequality comes from $\sum_{k = 1}^{\infty} \varepsilon_k \le c$,
thus
\begin{align}
& \sum_{k = 1}^K \NM{x_{\phi(k) + 1} - v_{\phi(k)}}{2}^2
\notag \\
& \le \frac{2}{\rho - L} \sum_{t = 0}^{q}
\left( \Delta_{t + 1} - \Delta_{K + t - q + 1} + \rho c \right) 
\notag \\
& \le \frac{2}{\rho - L} \sum_{t = 0}^{q}
\left( \Delta_{t + 1} - \inf F + \rho c \right) 
\notag \\
& \le \frac{2}{\rho - L} \sum_{t = 0}^{q}
\left( F(x_1) - \inf F + \rho c \right).
\label{eq:app7}
\end{align}
Let $c_2 = F(x_1) - \inf F + \rho c$.
As $c_2 < \infty$,
then
\begin{align}
\lim_{k \rightarrow \infty}
\NM{x_{\phi(k) + 1} - v_{\phi(k)}}{2} = 0,
\label{eq:app11}
\end{align}
which means $\lim_{k \rightarrow \infty} x_{\phi(k) + 1} = v_{\phi(k)}$,
and both sequence $\{ x_k \}$ and $\{ v_k \}$ are bounded
and both have at least one limit point.

Let $x_*$ be a limit point,
and $\{v_{k_j}\}$ be a subsequence of $\{ v_k \}$ such that $\lim_{k_j \rightarrow \infty} x_{k_j + 1}= \lim_{j \rightarrow \infty} v_{k_j} = x_*$.
As $\lim_{k_j \rightarrow \infty} \varepsilon_{k_j} = 0$, 
on step~11 of Algorithm~\ref{alg:ours}, we have
\begin{align*}
\lim_{k_j \rightarrow \infty} x_{k_j + 1} 
& = \lim_{k_j \rightarrow \infty} \text{inexact } \Px{\eta g}{z_{k_j}}
\\
& = \lim_{k_j \rightarrow \infty} \Px{\eta g}{z_{k_j}},
\\
& = \lim_{k_j \rightarrow \infty} \Px{\eta g}{v_{k_j} - \eta \nabla f(v_{k_j})},
\end{align*}
which shows 
\begin{align}
x_* = \Px{\eta g}{x_* - \eta \nabla f(x_*)}.
\label{eq:app16}
\end{align}
As $g$ is convex
and $f$ is Lipschitz-smooth,
$F$ is continuous,
then
\begin{align*}
\lim_{k_j \rightarrow \infty}
F(x_{k_j}) = F(x_*).
\end{align*} 
Thus, $x_*$ is also a critical point of \eqref{eq:compfunc} (Lemma~\ref{lem:prox}).

\subsection{Proposition~\ref{pr:inextpx}}

Let $v^*_{\phi(k)} = \Px{\eta g}{v_{\phi(k)} - \eta \nabla f(v_{\phi(k)})}$.
Note that $h_{\eta g}(x)$ is $1$-strongly convex and $0 \in \partial h_{\eta g}(v^*_{\phi(k)})$, 
thus
\begin{align*}
h(x_{\phi(k) + 1})
\ge h(v^*_{\phi(k)}) + \frac{1}{2}\NM{x_{\phi(k) + 1} - v^*_{\phi(k)}}{2}^2.
\end{align*}
As $h(x_{\phi(k) + 1}) - h(v^*_{\phi(k)}) \le \varepsilon_{\phi(k)}$,
thus 
$\NM{e_{\phi(k)}}{2}^2 
\equiv \NM{x_{k + 1} - v^*_{\phi(k)}}{2}^2 \le 2 \varepsilon_{\phi(k)}$.

\subsection{Proposition~\ref{pr:ietcvx}}

As
$\Gm{ v_{\phi(k)} } + e_k = x_{\phi(k) + 1} - v_{\phi(k)}$, 
thus 
\begin{align*}
\NM{\Gm{ v_{\phi(k)} }}{2}^2 
\le 2 \left( \NM{x_{\phi(k) + 1} - v_{\phi(k)}}{2}^2 + \NM{e_k}{2}^2 \right).
\end{align*}
Combining Proposition~\ref{pr:inextpx} and \eqref{eq:app11}, we have
\begin{align*}
\lim_{k \rightarrow \infty} \NM{ \Gm{v_{\phi(k)}} }{2}^2 = 0.
\end{align*}
Besides, using \eqref{eq:app7} and $\sum_{k}^{\infty} \varepsilon_k \le c$
\begin{align}
& \min_{k = 1, \dots, K}
\NM{\Gm{ v_{\phi(k)} }}{2}^2 
\notag \\
& \le \frac{2}{K} \sum_{k = 1}^K \left( \NM{x_{\phi(k) + 1} - v_{\phi(k)}}{2}^2 + \NM{e_k}{2}^2 \right)
\notag \\
& \le \frac{2}{K} \sum_{k = 1}^K \left( \NM{x_{\phi(k) + 1} - v_{\phi(k)}}{2}^2 + 2 \varepsilon_{\phi(k)} \right) 
\notag \\
&
\le \frac{2}{K} \left[ 4 c + \frac{1}{\rho - L} \sum_{t = 0}^{q}
\left( F(x_1) - \inf F + \rho c \right) \right].
\label{eq:app18}
\end{align}
As $c_1 = \max_{t = 1, \dots, q+1}F(x_t) - \inf F$, then \eqref{eq:app18} can be simplified as
\begin{align*}
\min_{k = 1, \dots, K}
\NM{\Gm{ v_{\phi(k)} }}{2}^2 
\le \frac{2}{K} \left[ 4 c + \frac{(q + 1)(c_1 + \rho c)}{\rho - L} \right] .
\end{align*}


\subsection{Theorem~\ref{the:temp1}}

Similar as Proposition~\ref{pr:app1},
we can obtain Proposition~\ref{pr:app3} here,
of which proof can also be done following Proposition~\ref{pr:app1}.

\begin{proposition} \label{pr:app3}
	$\Delta_{k + q + 1}
	\le \Delta_k
	- \frac{\delta}{2} \NM{x_{\phi(k) + 1} - v_{\phi(k)}}{2}^2$.
\end{proposition}

Summing over $k = 1, \dots, K$, then
\begin{align}
\sum_{t = 0}^{q}
\Delta_{t + 1} - \Delta_{K + t - q + 1}
\ge \frac{\delta}{2} \sum_{k = 1}^K \NM{x_{\phi(k) + 1} - v_{\phi(k)}}{2}^2.
\label{eq:app8}
\end{align}
As $\inf F > - \infty$,
let $K \rightarrow \infty$,
then 
\begin{align}
\sum_{t = 0}^{q}
\Delta_{t + 1} - \Delta_{K + t - q + 1}
& \le 
\sum_{t = 0}^{q}
(\Delta_{t + 1} - \inf F)
\label{eq:app9} \\
& \le (q + 1) c_1.
\notag
\end{align}
As $c_1 < \infty$ is a finite number,
from \eqref{eq:app8}, we have
\begin{align}
\lim_{k \rightarrow \infty}
\NM{x_{\phi(k) + 1} - v_{\phi(k)}}{2} = 0,
\label{eq:app13}
\end{align}
which means $\lim_{k \rightarrow \infty} x_{\phi(k) + 1} = v_{\phi(k)}$,
and both sequence $\{ x_k \}$ and $\{ v_k \}$ are bounded
and both have at least one limit point.

\subsection{Corollary~\ref{cor:temp1}}

Let $x_*$ be a limit point,
and $\{x_{k_j}\}$ be a subsequence of $\{ x_k \}$ such that
\begin{align*}
\lim_{k_j \rightarrow \infty} x_{k_j + 1}= \lim_{k_j \rightarrow \infty} v_{k_j} = x_*.
\end{align*}
Next,
we show that
\begin{align}
x_* = \Px{\eta}{x_* - \nabla f(x_*)}.
\label{eq:app12}
\end{align}
We will prove it by establishing contradictory.
Assuming \eqref{eq:app12} does not hold,
then $x_* \neq \Px{\eta}{x_* - \eta \nabla f(x_*)}$,
and we can find another
\begin{align*}
\tilde{x}_* = \lim_{k_j \rightarrow \infty} x_{k_j + 1} = \lim_{k_j \rightarrow \infty} \text{inexact } \Px{\eta g}{z_{k_j}} \neq x_*
\end{align*}
by Assumption~(i), 
which is in contradiction with $\lim_{k_j \rightarrow \infty} x_{k_j + 1} = x_*$.
As a result, 
we have 
\begin{align*}
x_* = \lim_{k_j \rightarrow \infty} x_{k_j + 1}
& = \lim_{k_j \rightarrow \infty} \text{inexact } \Px{\eta g}{z_{k_j}} 
\notag \\
& = \lim_{k_j \rightarrow \infty} \Px{\eta g}{z_{k_j}} 
\notag \\
& = \lim_{k_j \rightarrow \infty} \Px{\eta g}{v_{k_j} - \eta \nabla f(v_{k_j})} 
\notag \\
& = \Px{\eta}{x_* - \eta \nabla f(x_*)}.
\end{align*}
Besides,
by Assumption~(ii) we have
\begin{align*}
\lim_{k_j \rightarrow \infty} F(x_{k_j}) = F(x_*).
\end{align*}
Using Lemma~\ref{lem:prox} and \eqref{eq:app12}, 
$x_*$ is a critical point of \eqref{eq:compfunc}.

\subsection{Proposition~\ref{pr:ietncvx2}}

Using \eqref{eq:app13}, 
we have
\begin{align*}
\lim_{k \rightarrow \infty} \NM{ x_{\phi(k) + 1} - v_{\phi(k)} }{2}^2 = 0.
\end{align*}
Finally, using \eqref{eq:app8} and \eqref{eq:app9},
we have
\begin{align*}
& \min_{k = 1, \dots, K} \sum_{k = 1}^K \NM{x_{\phi(k) + 1} - v_{\phi(k)}}{2}^2
\\
& \le 
\frac{1}{K}\sum_{k = 1}^K \NM{x_{\phi(k) + 1} - v_{\phi(k)}}{2}^2
\\
& \le \frac{2}{\delta K} \sum_{t = 0}^{q}
\left( F(x_1) - \inf F \right)
= \frac{2 (q + 1) c_1}{\delta K}. 
\end{align*}

%

\end{document}